\documentclass[12pt]{amsart}
\usepackage{amsmath, amssymb, amsthm}
\usepackage{paralist, xcolor, tikz, hyperref}
\usepackage{young}
\usepackage[margin=1in,letterpaper,portrait]{geometry}
\usepackage{enumitem}
\usepackage{caption}
\usepackage[noadjust]{cite}
\usepackage{comment}
\usepackage{tabularx}
\usepackage{tikz,tikz-cd} 
\usepackage[normalem]{ulem}

\usetikzlibrary { decorations.pathmorphing, decorations.pathreplacing, decorations.shapes, shapes.geometric, calc}

\tikzstyle{vertex}=[
minimum size=0.15cm,
inner sep=0pt,
outer sep=0pt,
circle,
draw=black,
]
\tikzstyle{filled}=[
vertex,
fill=black
]
\tikzstyle{unfilled}=[
vertex,
fill=white
]
\tikzstyle{blank}=[
vertex,
fill=white,
draw=white
]

\tikzstyle{hourglass}=[
out=#1*20,
in={#1*20+180},
relative,
looseness=1.5
]

\usepackage{ytableau}
\usepackage[vcentermath,enableskew]{youngtab}

\newcommand{\permwebmap}{\Phi}
\newcommand{\permtabmap}{\pi}
\newcommand{\webtabmap}{\omega}

\renewcommand{\path}{\textsf{path}}

\newtheorem{conjecture}{Conjecture}[section]
\newtheorem{theorem}[conjecture]{Theorem}
\newtheorem{corollary}[conjecture]{Corollary}
\newtheorem{lemma}[conjecture]{Lemma}

\theoremstyle{definition}

\newtheorem{example}[conjecture]{Example}

\newtheorem*{question*}{Question}

\title{Forest webs and pattern avoidance}

\author[Jessica Striker]{Jessica Striker$^*$}
\address{Department of Mathematics, North Dakota State University, Fargo, ND, USA}
\email{jessica.striker@ndsu.edu}
\thanks{$^*$Research partially supported by NSF Grant DMS-2247089 and Simons Foundation Gift MP-TSM-00002802.}

\author[Bridget Eileen Tenner]{Bridget Eileen Tenner$^\dagger$}
\address{Department of Mathematical Sciences, DePaul University, Chicago, IL, USA}
\email{bridget@math.depaul.edu}
\thanks{$^\dagger$Research partially supported by NSF Grant DMS-2054436 and by a University Research Council Competitive Research Leave from DePaul University.}

\date{}

\begin{document}

\begin{abstract}
In a recent preprint, Mike Cummings showed that the smooth components of suitably parametrized Springer fibers are in bijection with contracted, fully reduced Pl\"ucker degree-two $\mathfrak{sl}_r$-webs of standard type and  that are forests. He showed these are enumerated by sequence A116731 in the OEIS, which is equinumerous with permutations avoiding the patterns $\{321,2143,3124\}$. Cummings posed the problem of strengthening this enumerative result by finding a bijection between these webs and a collection of pattern-avoiding permutations. Here we solve this problem, although notably not with the collection of patterns that Cummings had proposed. Rather, we give a bijection between this class of webs and permutations avoiding the patterns
$\{132,4321,3214\}$.
\end{abstract}

\maketitle

\section{Introduction}
At the \textit{Enumerative Combinatorics} workshop at the Mathematisches Forschungsinstitut Oberwolfach in January 2026, the first author gave a talk making an (admittedly controversial) claim that the best objects in enumerative combinatorics are \emph{webs}. She argued that the following important combinatorial objects are all webs in disguise: tableaux~\cite{GPPSS26,PatriasGOT,PPR2009}, noncrossing partitions~\cite{flamingo,jellyfish}, Catalan objects~\cite{GPPSS25}, alternating sign matrices~\cite{GPPSS26}, and plane partitions~\cite{GPPSS26}. The second author pointed out a serious omission in the list of important combinatorial objects: permutations. While permutation matrices are a subset of alternating sign matrices and may thus be interpreted as webs, this is not a very satisfying patch for the omission. In this note, we give a more substantial connection between webs and permutations by producing a bijection between webs whose associated Springer fibers have nice geometric properties, and a class of pattern-avoiding permutations. In addition to addressing the second author's objection, this answers a recent question of Mike Cummings~\cite{MC}.

Cummings was studying the geometry of \emph{Springer fibers}, which are fibers of resolutions of certain varieties indexed by partitions. (See \cite[Sec. 2.2-2.3]{MC} for background on Springer fibers.) It was known  that $\mathfrak{sl}_2$-webs (better known as noncrossing matchings) govern the geometry and topology of the Springer fibers of two-row partitions~\cite{Fung03}.  Cummings investigated the relation between Springer fibers  of two-column partitions and the \emph{Pl\"ucker degree-two $\mathfrak{sl}_r$-webs} studied in work of the first author with Gaetz, Pechenik, Pfannerer, and Swanson~\cite{GPPSS25}, building on work of Fraser~\cite{Fraser23}. These papers give a bijection between move equivalence classes of these webs and two-column rectangular standard Young tableaux. These tableaux also index components of the Springer fiber, and Fresse and Melnikov gave a tableau criterion that determines when the associated component of the Springer fiber has the appealing geometric property of being \emph{smooth}~\cite{FM}. One of Cummings's main theorems \cite[Theorem 3.2]{MC}  nicely reinterprets this tableau condition in terms of webs, showing that the Springer fiber component is smooth if and only if the corresponding web is a \emph{forest} (i.e., a graph with no cycles). He also gave an explicit enumeration of the smooth components of the Springer fiber using these forest webs~\cite[Corollary 3.4]{MC}, finding that his formula matched the enumeration of permutations avoiding the patterns $\{321,2143,3124\}$ given in \cite[A116731]{oeis}.  Cummings then asked for a bijective proof of this equinumerosity. Our main result gives such a bijection, but instead of using Cummings's proposed pattern avoidance class, we use the equinumerous class of permutations avoiding the patterns $\{132,4321,3214\}$.

In Section~\ref{sec:webs to tableaux}, we review the map of \cite{Fraser23,GPPSS25} 
from degree-two webs to two-column standard Young tableaux. These tableaux are Catalan objects, and one could give bijective maps between them and $p$-avoiding permutations for any $p\in S_3$. The main result of this work relies on an important combinatorial structure connected to the correspondence with $132$-avoiding permutations, in particular, and we give the desired bijection in Section~\ref{sec:tableaux to perms}. 
Cummings's work shows that the subset of such webs that are forests is characterized by having at most $3$ white vertices (see Lemma~\ref{lem:MC}). The  most basic degree-two  $\mathfrak{sl}_r$ webs with $4$ white vertices correspond to the permutations $4321$ and $3214$, and we show in Section~\ref{sec:webby forests} that those smallest cases, in fact, characterize the entire property of being a forest, thus proving our desired bijection. 
We conclude in Section~\ref{sec:enumeration} by using our result to recover the enumeration of \cite{MC}.

\section{From degree-two webs to two-column tableaux}\label{sec:webs to tableaux}
Webs are diagrams that represent polynomial invariants of certain spaces with algebraic structure, and they reduce complicated algebraic computations to diagrammatic manipulations. The webs of interest in this paper are \emph{$\mathfrak{sl}_r$}-webs, which represent polynomials such as the determinant that are invariant under the action of a matrix from the special linear group on the variables. See~\cite[Sec.~1 and 2.4]{MC}, for example, for more background on webs.

We refer to the particular webs that appear in this paper using the following, more specific, terminology of \cite{GPPSS25}.
We say an \emph{$r$-hourglass plabic graph of standard type and Pl\"ucker degree $d$}  is a planar properly bicolored graph $G$ embedded in a disk with:
\begin{itemize}
\item $dr$ boundary vertices all colored black and of degree one, 
\item  edges that each have a positive multiplicity, drawn (when greater than $1$) as a multiple edge with an hourglass twist 
\begin{tikzpicture}[baseline={([yshift=-.8ex]current bounding box.center)}]
    \draw (0,0) to[out=60,in=120] (.5,0) to[out=60,in=120] (1,0);
    \draw (0,0) to[out=20,in=160] (.5,0) to[out=20,in=160] (1,0);
    \draw (0,0) to[out=-20,in=-160] (.5,0) to[out=-20,in=-160] (1,0);
    \draw (0,0) to[out=-60,in=-120] (.5,0) to[out=-60,in=-120] (1,0);
    \fill (0,0) circle (2.5pt);
    \fill[white] (1,0) circle (2.5pt);
    \draw (1,0) circle (2.5pt);
\end{tikzpicture}, and
\item  the sum of the multiplicities of all the edges around any interior vertex equal to $r$. 
\end{itemize}
We consider $G$ up to
planar isotopy fixing the boundary circle.

The specific webs of interest are \emph{contracted, fully reduced $r$-hourglass plabic graphs of standard type and Pl\"ucker degree-two}, which we refer to as \emph{degree-two $\mathfrak{sl}_r$-webs} or \emph{degree-two webs} for short. We will not need the definitions of \emph{contracted} or \emph{fully reduced}, as the lemma below is sufficient for our purposes, and instead refer the interested reader to \cite[Sec.~2.2]{GPPSS25}.

\begin{lemma}[\protect{\cite[Lemmas 6.2--6.5]{GPPSS25}}]
\label{lem:CRG2}
    A degree-two $\mathfrak{sl}_r$-web has the following properties:
    \begin{itemize}
        \item each white interior vertex is adjacent to a boundary vertex,
        \item each black interior vertex has 3 incident hourglasses, and
        \item the number of interior white vertices exceeds the number of interior black vertices by exactly $2$.
    \end{itemize}
\end{lemma}

Cummings gave the following characterization of degree-two webs that are forests.  (He stated this in terms of \emph{claws}; the statement below is trivially equivalent.) 
\begin{lemma}[\protect{\cite[Lemma 3.1]{MC}}]
\label{lem:MC}
    A degree-two $\mathfrak{sl}_r$-web is a forest if and only if it has at most $3$ white vertices.
\end{lemma}

Note that by Lemma~\ref{lem:CRG2}, since the number of white vertices minus the number of black interior vertices is $2$, a degree-two forest web has either $3$ white vertices and $1$ interior black vertex or $2$ white vertices and no interior black vertices. See Figure~\ref{fig:forest_webs} for examples of forest webs and Example~\ref{ex:webs with 4 white vertices} for examples of non-forest webs.

We now review the map $\webtabmap$  from a degree-two $\mathfrak{sl}_r$-web $G$ to an $r\times 2$ standard Young tableau $T=\webtabmap(G)$, following \cite{Fraser23,GPPSS25}. (An \emph{$a\times b$ standard Young tableau} is a bijective filling of the partition shape $a^b$ with the numbers $\{1,\ldots,ab\}$ that is increasing across rows and down columns.) For clarity, we restrict our discussion to the case in which $G$ is a forest.
\begin{enumerate}
    \item Replace each hourglass of $G$ by a single weighted edge, with weight equal to the multiplicity.
    \item  Draw a polygon with the white vertices as the corners.
    \begin{itemize}
        \item 
    If there are three white vertices, the polygon is a triangle with the unique interior black vertex and adjacent hourglasses inside. Weight each edge of the triangle by the multiplicity of the hourglass that is not incident to the edge. Then delete the interior black vertex and incident hourglasses.
    \item If there are two white vertices, connect them by an edge of weight $r$, and consider this as the (degenerate) polygon. 
        \end{itemize}
    \item Construct a noncrossing matching by first turning all polygon edges 
    with weight $m$ into $m$ multiple (non-twisted) edges. Then delete the polygon vertices and connect each boundary edge to the corresponding edge on the other side of the white vertex to make the edges noncrossing.
    \item Finally, use the standard Catalan bijection to construct an $r\times 2$ standard Young tableau from this noncrossing matching. That is, for all connected pairs $i<j$, put $i$ in column $1$ and $j$ in column $2$ (and order the entries in each column to be increasing). Since it will be used later, we denote the noncrossing matching corresponding to a $r\times 2$ tableau $T$ as $\mu(T)$. 
\end{enumerate}
Figure~\ref{fig:forest_webs} gives two examples of this construction. Additional examples, and further details of the construction, can be found in~\cite[Sec.~2.5]{MC} or \cite[Sec.~2.4]{GPPSS25}.

\begin{figure}[htbp]
    \centering
$$\begin{minipage}{2in}\begin{tikzpicture}
    \node[draw=none,
        minimum size=3.6cm,
        regular polygon,
        regular polygon sides=8,
        regular polygon rotate=45
        ] (a) {};
    \node[draw=none,
        minimum size=4.1cm,
        regular polygon,
        regular polygon sides=8,
        regular polygon rotate=45
        ] (b) {};
    \foreach \x [evaluate=\x as \y using int(9-\x)] in {1,2,...,8}{
        \fill (a.corner \x) circle[radius=2.5pt];
        \draw (b.corner \x) node {$\y$};
        }
    \draw (0,0) circle (1.8cm);
    \fill (0,0) circle (2.5pt) coordinate (center);
    \draw ($(a.corner 8)!0.5!(a.corner 2)$) coordinate (17avg);
    \draw ($(17avg)!0.3!(center)$) coordinate (upper node);
    \draw (a.corner 2) -- (upper node) -- (a.corner 8);
    \draw (a.corner 1) -- (upper node) -- (center);
    \draw ($(a.corner 3)!0.5!(a.corner 5)$) coordinate (46avg);
    \draw ($(46avg)!0.3!(center)$) coordinate (lower node);
    \draw (a.corner 3) -- (lower node) -- (a.corner 5);
    \draw (a.corner 4) -- (lower node) -- (center);
    \draw ($(a.corner 6)!0.5!(a.corner 7)$) coordinate (23avg);
    \draw ($(23avg)!0.3!(center)$) coordinate (right inside);
    \draw (a.corner 6) -- (right inside) -- (a.corner 7);
    \draw ($(right inside)!0.5!(center)$) coordinate (right mid);
    \draw (center) to[out=30,in=150] (right mid) to[out=30,in=150] (right inside);
    \draw (center) to[out=-30,in=-150] (right mid) to[out=-30,in=-150] (right inside);
    \foreach \x in {upper node, lower node, right inside}{
        \fill[white] (\x) circle (2.5pt);
        \draw (\x) circle (2.5pt);
        }
\end{tikzpicture}\end{minipage}
\hspace{.15in}
    \begin{minipage}{.55in}\begin{ytableau}
        1 & 2\\
        3 & 4\\
        5 & 7\\
        6 & 8
    \end{ytableau}
    \end{minipage}
\hspace{.7in}
\begin{minipage}{2in}\begin{tikzpicture}
    \node[draw=none,
        minimum size=3.6cm,
        regular polygon,
        regular polygon sides=8,
        regular polygon rotate=45
        ] (a) {};
    \node[draw=none,
        minimum size=4.1cm,
        regular polygon,
        regular polygon sides=8,
        regular polygon rotate=45
        ] (b) {};
    \foreach \x [evaluate=\x as \y using int(9-\x)] in {1,2,...,8}{
        \fill (a.corner \x) circle[radius=2.5pt];
        \draw (b.corner \x) node {$\y$};
        }
    \draw (0,0) circle (1.8cm);
    \draw ($(a.corner 8)!0.5!(a.corner 3)$) coordinate (16avg);
    \draw ($(a.corner 7)!0.5!(a.corner 4)$) coordinate (25avg);
    \draw (a.corner 8) -- (16avg) -- (a.corner 3);
    \draw (a.corner 1) -- (16avg) -- (a.corner 2);
    \draw (a.corner 7) -- (25avg) -- (a.corner 4);
    \draw (a.corner 6) -- (25avg) -- (a.corner 5);
    \foreach \x in {16avg, 25avg}{
        \fill[white] (\x) circle (2.5pt);
        \draw (\x) circle (2.5pt);
        }
\end{tikzpicture}\end{minipage}
\hspace{.15in}
    \begin{minipage}{.55in}\begin{ytableau}
        1 & 2\\
        3 & 6\\
        4 & 7\\
        5 & 8
    \end{ytableau}
    \end{minipage}$$    
    \caption{Examples of the two types of forest webs and their corresponding tableaux. }
    \label{fig:forest_webs}
\end{figure}
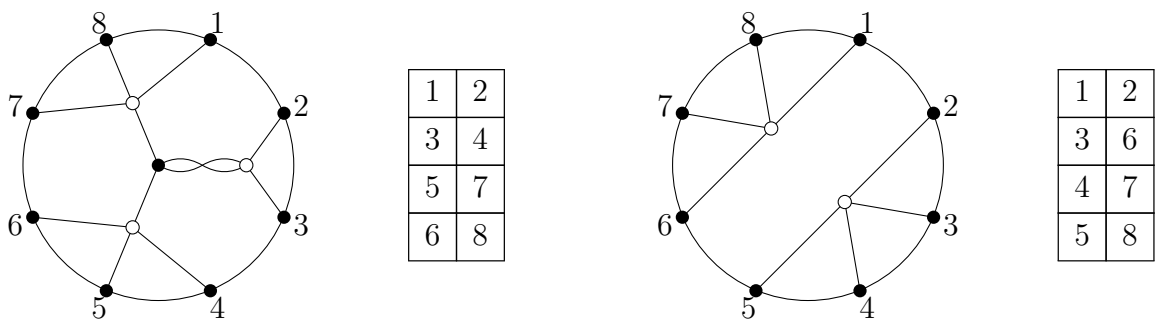

We show in the next section that the two-column tableaux produced by this map are in bijection with the set of $132$-avoiding permutations.

\section{From two-column tableaux to $132$-avoiding permutations}\label{sec:tableaux to perms}

Given permutations $w$ and $p$ written in one-line notation, we say that $w$ \emph{avoids the pattern} $p$ if there is no subsequence of $w$ that is in the same relative order as $p$. A permutation avoids a set of patterns if it avoids every pattern in that set. 
Rectangular standard Young tableaux with $2$ columns are a Catalan object, and are thus in bijection with $p$-avoiding permutations for any $p \in S_3$ \cite{catalan}. There are numerous bijective maps between these tableaux and such permutations. In this note, we identify a particular bijection with $132$-avoiding permutations, which will preserve the combinatorial structure we are most interested in. 
     
We say a \emph{Dyck path} from $(1,0)$ to $(r+1,r)$ is a path consisting of north and east steps that never passes below the line $y = x-1$. This variant is shifted one unit to the right from the more customary definition, so that the permutation we are defining can appear as a subset of $[1,r]^2$, rather than as a subset of $[0,r-1] \times [1,r]$.

Define a map $\permtabmap : \{r\times 2 \text{ standard Young tableaux}\} \rightarrow \{132\text{-avoiding permutations in } S_r\}$ by first sending a tableau $T$ to a  Dyck path $\path(T)$, where the $i$th step is north if and only if $i$ is in the first column of $T$, and then using $\path(T)$ to define a $132$-avoiding permutation as follows, and demonstrated in Example~\ref{ex:example in S_11} and Figure~\ref{fig:example in S_11}.  We follow the construction of \cite{AR2003}, which is equivalent to the map of \cite{CK2001}.

\begin{figure}[bp]
    T = \begin{minipage}{.5in}\begin{ytableau}
        1 & 3\\
        2 & 7\\
        4 & 8\\
        5 & 9\\
        6  & 11\\
        10 & 12\\
        13 & 15\\
        14 & 18\\
        16 & 19\\
        17 & 20\\
        21 & 22
    \end{ytableau}\end{minipage} 
    \hspace{.3in}
    \begin{minipage}{2.25in}\begin{tikzpicture}[scale=.5]
        \begin{scope}
        \clip (.75,0) rectangle (12,11.25);
        \foreach \x in {(1,2),(2,5),(5,6),(7,8),(8,10),(11,11)} {
            \fill[black!20] \x++(-.25,-.25) rectangle ++(11.25,.5);
            \fill[black!20] \x++(-.25,-.25) rectangle ++(.5,-11.25);
            \draw[thick] \x circle (5pt);
            }
        \end{scope}
        \draw (5,7.5) node {$\path(T)$};
        \draw[thick,dotted] (1,0) -- (12,11);
        \draw[black!60] (1,0) grid (12,11);
        \draw[ultra thick] (1,0) -- (1,2) -- (2,2) -- (2,5) -- (5,5) -- (5,6) -- (7,6) -- (7,8) -- (8,8) -- (8,10) -- (11,10) -- (11,11) -- (12,11);
        \foreach \x in {(1,2),(2,5),(5,6),(7,8),(8,10),(11,11)} {
            \draw[thick] \x circle (5pt);
        }
        \draw (7,-.5) node[below] {$y = x-1$};
        \draw[->] (7,-.5) -- (4.25,2.75);
    \end{tikzpicture}\end{minipage}
    \hspace{.3in}
    \begin{minipage}{2.25in}\begin{tikzpicture}[scale=.5]
        \begin{scope}
        \clip (.75,0) rectangle (12,11.25);
        \foreach \x in {(1,2),(2,5),(5,6),(7,8),(8,10),(11,11),
        (3,4), (9,9),
        (4,3), (10,7),
        (6,1)
        } {
            \fill[black!20] \x++(-.25,-.25) rectangle ++(11.25,.5);
            \fill[black!20] \x++(-.25,-.25) rectangle ++(.5,-11.25);
            \draw[thick] \x circle (5pt);
            }
        \end{scope}
        \draw[black!60] (1,0) grid (12,11);
        \draw[ultra thick] (1,0) -- (1,2) -- (2,2) -- (2,5) -- (5,5) -- (5,6) -- (7,6) -- (7,8) -- (8,8) -- (8,10) -- (11,10) -- (11,11) -- (12,11);
        \foreach \x in {(2,5),(5,6),(7,8),(8,10),(11,11),
        (3,4), (9,9),
        (4,3), (10,7),
        (6,1)
        } {
            \draw[thick] \x circle (5pt);
        }
        \draw[white] (7,-.5) node[below] {$y = x-1$};
    \end{tikzpicture}\end{minipage}
    \caption{A two-column tableau, its corresponding Dyck path, and the resulting permutation that is constructed by $\permtabmap$.}\label{fig:example in S_11}
\end{figure}
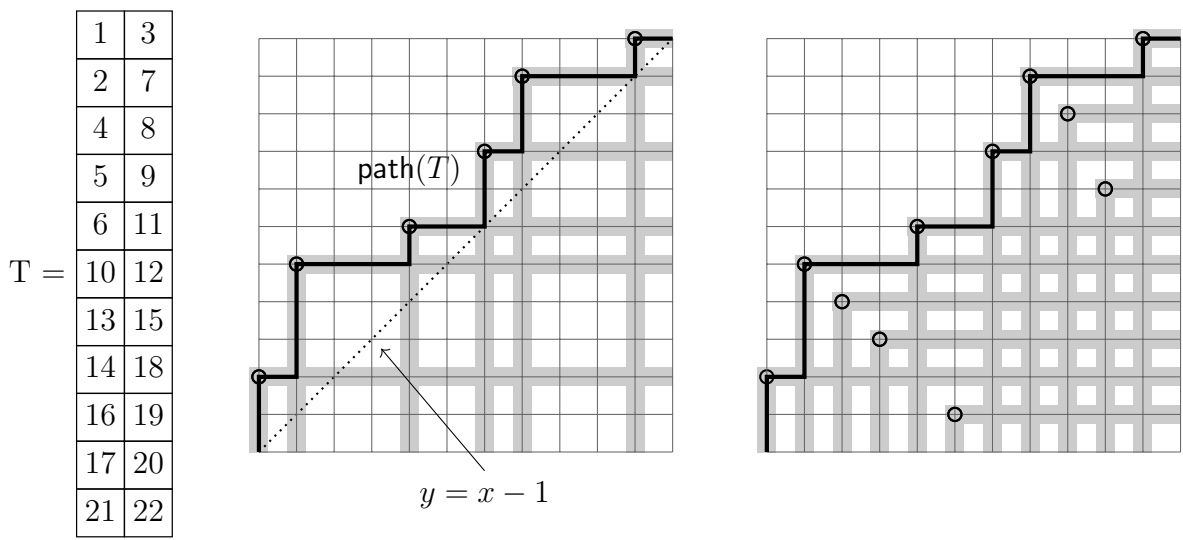

\begin{enumerate}
    \item Put a marker at each northwest corner of $\path(T)$.
    \item Draw a ray eastward and a ray southward from each marker.
    \item Put a marker at any northwest-most point of $[1,r]^2$ that lies under $\path(T)$ and that has not yet been marked by a marker or a ray; extend rays eastward and southward from this newly marked point. Repeat this process iteratively with the remaining unmarked points of $[1,r]^2$ that lie under $\path(T)$.
    \item This process terminates after $r$ markers are placed, and those $r$ points determine a permutation $w$. Read the permutation from the top downward (``matrix coordinate''-style): a marker at $(x,y)$ means that $w(r+1-y) = x$.  
    \item Let $\permtabmap(T)$ be this permutation $w$. That $\permtabmap$ is a bijection follows from \cite{CK2001}, and this construction ensures that $w$ is $132$-avoiding.
\end{enumerate}

\begin{example}\label{ex:example in S_11}
    The tableau $T$ in Figure~\ref{fig:example in S_11} produces the path $\path(T)$ indicated, with northwest corners circled and eastward and southward rays drawn from those circled points. The rightmost figure has all such points and rays indicated, from which we see that $\permtabmap(T) = 11 \ 8 \ 9 \ 7 \ 10 \ 5 \ 2 \ 3 \ 4 \ 1 \ 6$. This permutation is $132$-avoiding, as predicted.
\end{example}

\section{From forest webs to $\{132,4321,3214\}$-avoiding permutations}
\label{sec:webby forests}

Let $\permwebmap$ be the composition of maps $\permtabmap \circ \webtabmap$, so $\permwebmap$ is a map from degree-two webs to $132$-avoiding permutations. 
Our main result is that when we restrict the domain to be webs that are forests, this $\permwebmap$ is a bijection between that set and $\{132,4321,3214\}$-avoiding permutations.

As shown by Cummings (see Lemma~\ref{lem:MC} above), the web will be a forest if and only if there are fewer than four white vertices \cite{MC}. To get a sense of what this means for the map $\permwebmap$, consider the simplest examples of webs that fail to meet this criterion.

\begin{example}\label{ex:webs with 4 white vertices}
    The two $4\times 2$ tableaux whose webs have four white vertices are given below, along with their corresponding webs.
$$
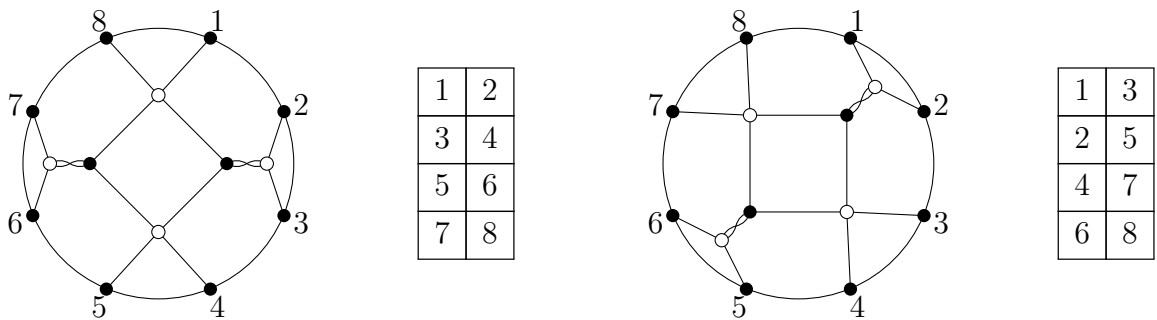
\begin{minipage}{2in}\begin{tikzpicture}
    \node[draw=none,
        minimum size=3.6cm,
        regular polygon,
        regular polygon sides=8,
        regular polygon rotate=45
        ] (a) {};
    \node[draw=none,
        minimum size=4.1cm,
        regular polygon,
        regular polygon sides=8,
        regular polygon rotate=45
        ] (b) {};
    \node[draw=none,
        minimum size=1.8cm,
        regular polygon,
        regular polygon sides=4,
        regular polygon rotate=45
        ] (c) {};
    \foreach \x [evaluate=\x as \y using int(9-\x)] in {1,2,...,8}{
        \fill (a.corner \x) circle[radius=2.5pt];
        \draw (b.corner \x) node {$\y$};
        }
    \draw (0,0) circle (1.8cm);
    \draw (a.corner 1) -- (c.corner 1) -- (c.corner 4) -- (c.corner 3) -- (a.corner 4);
    \draw (a.corner 8) -- (c.corner 1) -- (c.corner 2) -- (c.corner 3) -- (a.corner 5);
    \draw ($(a.corner 2)!0.5!(a.corner 3)$) coordinate (67avg);
    \draw ($(67avg)!0.3!(c.corner 2)$) coordinate (left inside);
    \draw ($(a.corner 6)!0.5!(a.corner 7)$) coordinate (23avg);
    \draw ($(23avg)!0.3!(c.corner 4)$) coordinate (right inside);
    \draw (a.corner 2) -- (left inside) -- (a.corner 3);
    \draw (a.corner 6) -- (right inside) -- (a.corner 7);
    \draw ($(left inside)!0.5!(c.corner 2)$) coordinate (left mid);
    \draw ($(right inside)!0.5!(c.corner 4)$) coordinate (right mid);
    \draw (left inside) to[out=30,in=150] (left mid) to[out=30,in=150] (c.corner 2);
    \draw (left inside) to[out=-30,in=-150] (left mid) to[out=-30,in=-150] (c.corner 2);
    \draw (c.corner 4) to[out=30,in=150] (right mid) to[out=30,in=150] (right inside);
    \draw (c.corner 4) to[out=-30,in=-150] (right mid) to[out=-30,in=-150] (right inside);
    \foreach \x in {left inside, right inside}{
        \fill[white] (\x) circle (2.5pt);
        \draw (\x) circle (2.5pt);
        }
    \foreach \x in {1,3}{
        \fill[white] (c.corner \x) circle (2.5pt);
        \draw (c.corner \x) circle (2.5pt);}
    \foreach \x in {2,4}
        \fill (c.corner \x) circle (2.5pt);
\end{tikzpicture}\end{minipage}
\hspace{.2in}
    \begin{minipage}{.55in}\begin{ytableau}
        1 & 2\\
        3 & 4\\
        5 & 6\\
        7 & 8
    \end{ytableau}
    \end{minipage}
    \hspace{.6in}
\begin{minipage}{2in}\begin{tikzpicture}
    \node[draw=none,
        minimum size=3.6cm,
        regular polygon,
        regular polygon sides=8,
        regular polygon rotate=45
        ] (a) {};
    \node[draw=none,
        minimum size=4.1cm,
        regular polygon,
        regular polygon sides=8,
        regular polygon rotate=45
        ] (b) {};
    \node[draw=none,
        minimum size=1.8cm,
        regular polygon,
        regular polygon sides=4,
        regular polygon rotate=90
        ] (c) {};
    \foreach \x [evaluate=\x as \y using int(9-\x)] in {1,2,...,8}{
        \fill (a.corner \x) circle[radius=2.5pt];
        \draw (b.corner \x) node {$\y$};
        }
    \draw (0,0) circle (1.8cm);
    \draw (a.corner 1) -- (c.corner 1) -- (c.corner 4) -- (c.corner 3) -- (a.corner 5);
    \draw (a.corner 2) -- (c.corner 1) -- (c.corner 2) -- (c.corner 3) -- (a.corner 6);
    \draw ($(a.corner 4)!0.5!(a.corner 3)$) coordinate (56avg);
    \draw ($(56avg)!0.3!(c.corner 2)$) coordinate (left inside);
    \draw ($(a.corner 8)!0.5!(a.corner 7)$) coordinate (12avg);
    \draw ($(12avg)!0.3!(c.corner 4)$) coordinate (right inside);
    \draw (a.corner 4) -- (left inside) -- (a.corner 3);
    \draw (a.corner 8) -- (right inside) -- (a.corner 7);
    \draw ($(left inside)!0.5!(c.corner 2)$) coordinate (left mid);
    \draw ($(right inside)!0.5!(c.corner 4)$) coordinate (right mid);
    \draw (left inside) to[out=75,in=-165] (left mid) to[out=75,in=-165] (c.corner 2);
    \draw (left inside) to[out=15,in=-105] (left mid) to[out=15,in=-105] (c.corner 2);
    \draw (c.corner 4) to[out=75,in=-165] (right mid) to[out=75,in=165] (right inside);
    \draw (c.corner 4) to[out=15,in=-105] (right mid) to[out=15,in=-105] (right inside);
    \foreach \x in {left inside, right inside}{
        \fill[white] (\x) circle (2.5pt);
        \draw (\x) circle (2.5pt);
        }
    \foreach \x in {1,3}{
        \fill[white] (c.corner \x) circle (2.5pt);
        \draw (c.corner \x) circle (2.5pt);}
    \foreach \x in {2,4}
        \fill (c.corner \x) circle (2.5pt);
\end{tikzpicture}\end{minipage}
\hspace{.2in}
    \begin{minipage}{.55in}\begin{ytableau}
        1 & 3\\
        2 & 5\\
        4 & 7\\
        6 & 8
    \end{ytableau}
    \end{minipage}
$$
Under the map $\permtabmap$, these tableaux correspond to the permutations $4321$ and $3214$, respectively.
\end{example}

Example~\ref{ex:webs with 4 white vertices} foreshadows our main result. We preface that theorem with two helpful lemmas, with the first involving the shape of a  Dyck path.

\begin{lemma}\label{lem:n1 arc means weakly above y = x}
    A {noncrossing matching} with $2r$ vertices has an arc between $2r$ and $1$ if and only if all but the first and last step of its corresponding Dyck path (what we will call its ``interior'') lives weakly above the line $y = x$.
\end{lemma}

\begin{proof}
    Let $T$ be an $r\times 2$ tableau with noncrossing matching $\mu(T)$. Let $q \in [2,2r]$ be connected to $1$ by an arc in $\mu(T)$, and separate the arcs of the matching into those involving vertices $[1,q]$ and those involving vertices $[q+1,2r]$. This fails to be $2$-part partition of the arcs if and only if $q = 2r$. This is equivalent to every topmost $m\times 2$ subtableau of $T$ containing at least one value larger than $2m$, for all $m < 2r$, which is equivalent to the interior of $\path(T)$ lying weakly above $y = x$.
\end{proof}

The second preliminary lemma refers to the construction of the permutation $\permtabmap(T)$.

\begin{lemma}\label{lem:points are southeast of each other}
    Let $T$ be a two-column tableau. Any two points of $\permtabmap(T)$ lying southeast of a given marker on $\path(T)$ must be such that one is southeast of the other.
\end{lemma}

\begin{proof}
    Suppose, for the purpose of obtaining a contradiction, that this is not the case. Then, as shown below, there is a marker on $\path(T)$ with two markers to its southeast, one of which is southwest of the other. 
    $$\begin{tikzpicture}[scale=.5]
        \draw[ultra thick] (0,2) -- (0,4) -- (4,4);
        \foreach \x in {(0,4), (2,1), (5,2)}
            {\draw \x circle (5pt);}
        \foreach \x in {(2,1), (5,2)}
            {\fill \x circle (3pt);}
        \draw[ultra thick,decorate,decoration={coil,aspect=0}] (0,2) -- (-1.5,.5);
        \draw[ultra thick,decorate,decoration={coil,aspect=0}] (4,4) -- (6.5,5.5);
        \draw (-1.5,2.5) node[left] {$\path(T)$};
        \draw[black!60] (-1.5,.5) grid (6.5,5.5);
    \end{tikzpicture}$$
    These three points identify a $132$-pattern in $\permtabmap(T)$, which is a contradiction.
\end{proof}

We are now ready to prove our main result. 

\begin{theorem}\label{thm:webby forests as pattern avoidance}
    The map $\permwebmap$ is a bijection between 
    forest degree-two $\mathfrak{sl}_r$-webs
    and the set of $\{132,4321,3214\}$-avoiding permutations in $S_r$. 
    \end{theorem}

\begin{proof}
    We will show that the image of the set of forest degree-two $\mathfrak{sl}_r$-webs under the bijection $\permwebmap$ is exactly the set of $\{132, 4321, 3214\}$-avoiding permutations in $S_r$. The $132$-avoidance is guaranteed by $\permtabmap$, so we need only consider $4321$- and $3214$-patterns.

    For the first direction, suppose that we have a web that is not a  forest and $T$ its corresponding two-column tableau via the map $\webtabmap$. The forest condition means, by Lemma \ref{lem:MC}, that the web has at least four white vertices, which can arise from two scenarios.

    In the first case, those white vertices correspond to short arcs $\{a,a+1\}$, $\{b,b+1\}$, $\{c,c+1\}$, and $\{d,d+1\}$, where $1 \le a$, $a+1 < b$, $b+1 < c$, and $c+1 < d$ in the noncrossing matching $\mu(T)$. Then
    the web corresponds, via $\webtabmap$, to a tableau and a Dyck path as shown in Figure~\ref{fig:four corners in the path}. 
    \begin{figure}[htbp]
    \begin{minipage}{.6in}\begin{tikzpicture}[scale=.6]
        \draw (-.5,0) rectangle (1.5,10.5);
        \draw (.5,0) -- (.5,10.5);
        \draw (0,10) node {$1$}; \draw (-.5,9.5) -- (.5,9.5);
        \draw (1,.5) node {$n$}; \draw (.5,1) -- (1.5,1);
        \draw (0,8) node {$a$}; \draw (1,8.75) node {$\scriptstyle a+1$};
        \draw (0,6) node {$b$}; \draw (1,6.75) node {$\scriptstyle b+1$};
        \draw (0,4) node {$c$}; \draw (1,4.75) node {$\scriptstyle c+1$};
        \draw (0,2) node {$d$}; \draw (1,2.75) node {$\scriptstyle d+1$};
        \foreach \x in {1,3,5,7,9} {\draw (0,\x)++(0,.2) node {$\vdots$}; \draw (1,\x)++(0,1) node {$\vdots$};}
    \end{tikzpicture}\end{minipage}
    \hspace{.75in}
    \begin{minipage}{3.75in}
    \begin{tikzpicture}[scale=.5]
        \draw[black!60] (1,0) grid (14,13);
        \foreach \x in {(3,3), (6,5), (8,9), (11,11)} 
            {\draw[ultra thick] \x --++(0,1) --++(1,0);
            \draw[thick] \x++(0,1) circle (5pt);}
        \draw (2.9,3.5) node[left] {step $a$};
        \draw (5.8,5.5) node[left] {step $b$};
        \draw (7.8,9.5) node[left] {step $c$};
        \draw (10.9,11.5) node[left] {step $d$};
        \draw[ultra thick,decorate,decoration={coil,aspect=0}] (1,1) -- (3,3);
        \draw[ultra thick,decorate,decoration={coil,aspect=0}] (4,4) -- (6,5);
        \draw[ultra thick,decorate,decoration={coil,aspect=0}] (7,6) -- (8,9);
        \draw[ultra thick,decorate,decoration={coil,aspect=0}] (9,10) -- (11,11);
        \draw[ultra thick,decorate,decoration={coil,aspect=0}] (12,12) -- (13,13);
        \draw[ultra thick] (1,0) -- (1,1);
        \draw[ultra thick] (13,13) -- (14,13);
        \draw[thick,dotted] (1,0) -- (14,13);
        \draw (15,6.5) node[right] {$y = x-1$};
        \draw[->] (15,6.5) -- (9.75,8.5);
    \end{tikzpicture}
    \end{minipage}
    \caption{A tableau and path corresponding to a web having four short arcs of the form $\{i,i+1\}$.}\label{fig:four corners in the path}
    \end{figure}
    Therefore the corresponding permutation, via $\permtabmap$, has a $4321$-pattern.

    If this is not the case, then there are exactly four short arcs in $\mu(T)$, and they have the form 
    $\{2r,1\}$, $\{b,b+1\}$, $\{c,c+1\}$, and $\{d,d+1\}$, for $1 < b$, $b+1 < c$, $c+1 < d$, and $d+1 < 2r$. 
    Then the corresponding path $\path(T)$ has the form shown in Figure~\ref{fig:three corners in the path and above the line}, 
    \begin{figure}[htbp]
    \begin{tikzpicture}[scale=.5]
        \fill[black!20] (6.5,.5) rectangle (11.5,4.5);
        \draw[black!60] (1,0) grid (12,11);
        \foreach \x in {(1,3), (4,7), (7,10)} 
            {\draw[ultra thick] \x --++(0,1) --++(1,0);
            \draw[thick] \x++(0,1) circle (5pt);}
        \draw (1,3.5) node[left] {step $b$};
        \draw (3.8,7.5) node[left] {step $c$};
        \draw (6.9,10.5) node[left] {step $d$};
        \draw[ultra thick] (1,0) -- (1,4) -- (4,4) -- (4,8) -- (7,8) -- (7,11) -- (12,11);
        \draw[ultra thick] (10,11) -- (12,11);
        \draw[thick,dotted] (1,1) -- (11,11);
        \draw[red] (5,3) node {$\times$};
        \draw[red] (10,6) node {$\times$};
        \draw (13,6.5) node[right] {$y = x$};
        \draw[->] (13,6.5) -- (8.75,8.5);
    \end{tikzpicture}
    \caption{A tableau and path corresponding to a web having four short arcs, one of which is $\{2r,1\}$.}\label{fig:three corners in the path and above the line}
    \end{figure}
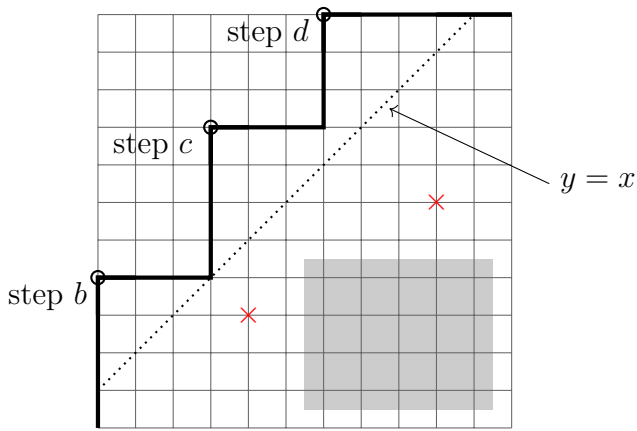
    where the interior of $\path(T)$ lives weakly above $y = x$, by Lemma~\ref{lem:n1 arc means weakly above y = x}. 
    
    If a marker (i.e., a point of $\permtabmap(T)$) lies in the shaded region of Figure~\ref{fig:three corners in the path and above the line}, then $\permtabmap(T)$ has a $3214$-pattern. Suppose, for the purpose of obtaining a contradiction, that there is no such marker. 
    Because $b \ge 2$, there is at least one point of $\permtabmap(T)$ below step $b$. Thus this must appear to the left of the shaded region, and, because the interior of $\path(T)$ lies weakly above $y = x$, there is such a point to the right of step $c$. Let the leftmost red $\textcolor{red}{\times}$ indicate this point in the diagram.
    This, then, implies that there is a point to the right of step $d$ that is below step $c$ (necessarily above the shaded region), indicated by the rightmost red $\textcolor{red}{\times}$ in the diagram. 
    Those two points and the marker after step $c$ would form a $132$-pattern in $\permtabmap(T)$, which is impossible. Thus a point of $\permtabmap(T)$ lies in the shaded region of the figure, and hence $\permtabmap(T)$ has a $3214$-pattern.

    To show the converse statement, suppose that $w$ is a $132$-avoiding permutation that contains at least one of the patterns $\{4321,3214\}$. Let $T = \permtabmap^{-1}(w)$ be the corresponding two-column tableau. 
    In order to avoid $132$ and yet have one of these patterns, it follows from Lemma~\ref{lem:points are southeast of each other} that there must be at least three markers along $\path(T)$. If there are four markers along $\path(T)$ then the corresponding web $\permwebmap^{-1}(w)$ has at least four white vertices and hence is not a forest \cite{MC}. 
    Consider, now, that there are only three such markers, as shown in Figure~\ref{fig:converse argument}. By Lemma~\ref{lem:points are southeast of each other}, there are no $4321$-patterns in $w$, so $w$ must have a $3214$-pattern. This and Lemma~\ref{lem:points are southeast of each other} force a point of the permutation to lie east of point $D$ and south of point $B$, as indicated by $E$ in Figure~\ref{fig:converse argument}. 
    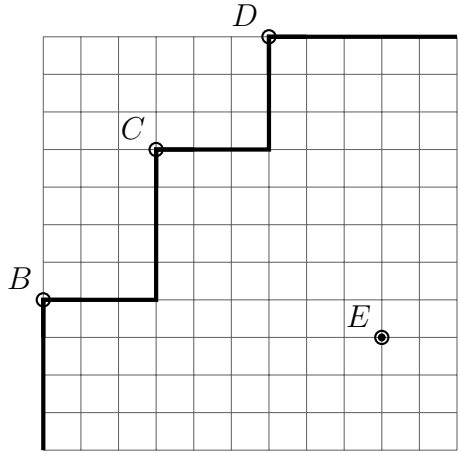
\begin{figure}[htbp]
    \begin{tikzpicture}[scale=.5]
        \draw[black!60] (1,0) grid (12,11);
        \foreach \x in {(1,3), (4,7), (7,10)} 
            {\draw[ultra thick] \x --++(0,1) --++(1,0);
            \draw[thick] \x++(0,1) circle (5pt);}
        \draw (1,4) node[above left] {$B$};
        \draw (4,8) node[above left] {$C$};
        \draw (7,11) node[above left] {$D$};
        \draw[ultra thick] (1,0) -- (1,4) -- (4,4) -- (4,8) -- (7,8) -- (7,11) -- (12,11);
        \draw[thick] (10,3) circle (5pt);
        \fill (10,3) circle (3pt);
        \draw (10,3) node[above left] {$E$};
    \end{tikzpicture}
    \caption{A path with three northwest corners whose corresponding permutation has a $3214$-pattern.}\label{fig:converse argument}
    \end{figure}
    The existence of the point $E$ forces the interior of $\path(T)$ to lie weakly above the line $y = x$, because otherwise the eastward and southward rays from the points of $\permtabmap(T)$ would not cover the region below $\path(T)$. By Lemma~\ref{lem:n1 arc means weakly above y = x}, this means that the web has an arc between $2r$ and $1$, which indicates a fourth white vertex in the web, meaning that we had not started with a  forest web.
\end{proof}

This result appears as entry \cite[P0072]{dppa}.

\section{Enumeration}\label{sec:enumeration}

From Theorem~\ref{thm:webby forests as pattern avoidance}, we can recover Cummings's enumeration result.

\begin{corollary}[{cf.~\cite{MC}}]\label{cor:enumeration}
    There are $r + 2\binom{r}{3}$ forest degree-two $\mathfrak{sl}_r$-webs. 
\end{corollary}

\begin{proof}
    By Theorem~\ref{thm:webby forests as pattern avoidance}, it suffices to enumerate $\{132,4321,3214\}$-avoiding permutations in $S_r$. Let $w$ be such a permutation, and set $m := w^{-1}(r)$. 
    
    To avoid $132$, every element of $\{w(1),\ldots,w(m-1)\}$ must be larger than every element of $\{w(m+1),\ldots,w(r)\}$. Let $u \in S_{m-1}$ be the permutation that is order isomorphic to $w(1)\cdots w(m-1)$, and let $v \in S_{r-m}$ be order isomorphic to $w(m+1)\cdots w(r)$. Both $u$ and $v$ must avoid $132$. They must also avoid $321$: $u$ to avoid $3214$ in $w$, and $v$ to avoid $4321$ in $w$. Finally, at most one of $\{u,v\}$ can have a descent, to avoid $4321$ in $w$. This characterizes $w$ completely, and we can enumerate such permutations using the following scheme.
    $$\begin{minipage}{1.35in}
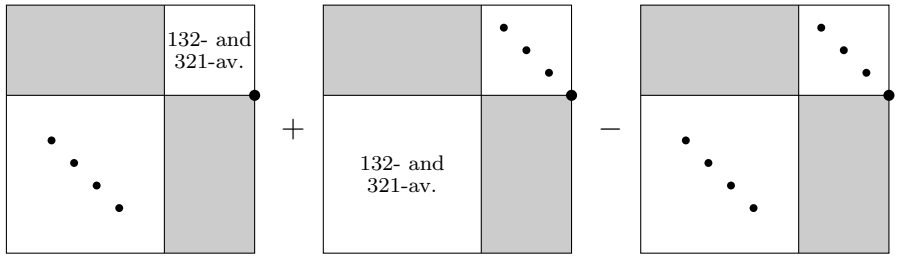
\begin{tikzpicture}[scale=.3]
        \fill[black!20] (1,8) rectangle (8,12);
        \fill[black!20] (8,1) rectangle (12,8);
        \draw (1,1) rectangle (12,12);
        \draw (1,8) -- (12,8);
        \draw (8,1) -- (8,12);
        \fill (12,8) circle (7pt);
        \foreach \x in {3,...,6}
            {\fill (\x,9-\x) circle (5pt);]}
        \draw (10,10.5) node {{\tiny $132$- and}};
        \draw (10,9.5) node {{\tiny $321$-av.}};
    \end{tikzpicture}\end{minipage} 
    \ + \  
    \begin{minipage}{1.35in}\begin{tikzpicture}[scale=.3]
        \fill[black!20] (1,8) rectangle (8,12);
        \fill[black!20] (8,1) rectangle (12,8);
        \draw (1,1) rectangle (12,12);
        \draw (1,8) -- (12,8);
        \draw (8,1) -- (8,12);
        \fill (12,8) circle (7pt);
        \foreach \x in {9,10,11}
            {\fill (\x,20-\x) circle (5pt);]}
        \draw (4.5,5) node {{\tiny $132$- and}};
        \draw (4.5,4) node {{\tiny $321$-av.}};
    \end{tikzpicture}\end{minipage}
    \ - \ 
    \begin{minipage}{1.35in}\begin{tikzpicture}[scale=.3]
        \fill[black!20] (1,8) rectangle (8,12);
        \fill[black!20] (8,1) rectangle (12,8);
        \draw (1,1) rectangle (12,12);
        \draw (1,8) -- (12,8);
        \draw (8,1) -- (8,12);
        \fill (12,8) circle (7pt);
        \foreach \x in {3,...,6}
            {\fill (\x,9-\x) circle (5pt);]}
        \foreach \x in {9,10,11}
            {\fill (\x,20-\x) circle (5pt);]}
    \end{tikzpicture}\end{minipage}$$
There are $\binom{k}{2} + 1$ permutations in $S_k$ that avoid both $321$ and $132$. Therefore the number of permutations in $S_r$ that avoid $\{132,4321,3214\}$ is
$$\sum_{m=1}^r \left( \binom{m-1}{2} + 1 + \binom{r-m}{2} + 1 - 1\right)
= \binom{r}{3} + r + \binom{r}{3} = r + 2\binom{r}{3},$$
by the hockey stick identity.\footnote{The result is known by many names \cite{hockey_stick}. We use this one in honor of winter in Canada, the home of Mike Cummings whose paper~\cite{MC} inspired this note, and North Dakota, which was the setting of a lovely and productive research visit for the authors, including a frigid hike, great karaoke, and the work behind this note.}
\end{proof}

\section*{Acknowledgments}
The authors thank Oliver Pechenik for helpful comments and Mike Cummings for posing such a lovely problem relating webs and permutations. They also thank the developers of SageMath~\cite{sage}, 
which was useful in this research.

\end{document}